\documentclass[english]{amsart}

\usepackage{amsmath}
\usepackage{amssymb}
\usepackage{amsthm}
\usepackage{amscd}
\usepackage{babel}
\usepackage[latin1]{inputenc}
\usepackage{graphicx}

\newtheorem{prop}{Proposition}

\newtheorem{lem}{Lemma}
\newtheorem{defi}{Definition}

\renewcommand{\subjclassname}{AMS \textup{2010} Mathematics Subject
Classification\ }

\author{Jos\'{e} Mar\'{i}a Grau}
\address{Departamento de Matemáticas, Universidad de Oviedo\\ Avda. Calvo Sotelo, s/n, 33007 Oviedo, Spain}
\email{grau@uniovi.es}
\author{Antonio M. Oller-Marc\'{e}n}
\address{Departamento de Matem\'{a}ticas, Universidad de Zaragoza\\ C/Pedro Cerbuna 12, 50009 Zaragoza, Spain} \email{oller@unizar.es}
\title{Giuga Numbers and the arithmetic derivative}

\begin{document}
\maketitle

\begin{abstract} We characterize Giuga Numbers as solutions to the equation $n'=an+1$, with $a \in
\mathbb{N}$ and $n'$ being the arithmetic derivative. Although this fact does not refute Lava's conjecture, it brings doubts about its veracity. 
\end{abstract}
\subjclassname{11B99, 11A99}


\section{Introduction }
\subsection{The Arithmetic Derivative.}

The arithmetic derivative was introduced by Barbeau in \cite{BAR}
(see A003415 in the \emph{On-Line Encyclopedia of Integer
Sequences}). The derivative of an integer is defined to be the unique map sending
every prime to 1 and satisfying the Leibnitz rule; i.e: $p'=1$, for any
prime $p$, and $(nm)'=nm'+mn'$ for any $n,m\in \mathbb{N}$. This map makes
sense and is well-defined (see \cite{VIC}).

\begin{prop} 
If $n=\prod_{i=1}^{k}p_i^{r_i}$ is the factorization of $n$ in prime
powers, then the only way to define $n'$ satisfying the desired properties is:
$$ n'=n\sum_{i=1}^{k}\frac{r_i}{p_i}.$$
\end{prop}

\subsection{Giuga Numbers} 

In \cite{BOR}, Giuga Numbers were introduced in the following way motivated by previous work by Giuga \cite{GIU}.

\begin{defi}
A Giuga Number is a composite number $n$ such that $p$ divides $\frac{n}{p}-1$ for
every $p$, prime divisor of $n$.
\end{defi}

It follows easily from definition that every Giuga Number is square-free. There are several characterizations of Giuga Numbers, the most important
being the following.

\begin{prop}
Let $n$ be a composite integer. Then, the following are equivalent:
\begin{itemize}
\item [i)] $n$ is a Giuga Number.
\item [ii)] $\displaystyle{\sum_{p|n} \frac{1}{p}-\prod_{p|n} \frac{1}{p}
\in \mathbb{N}}$ (see \cite{GIU}).
\item[iii)] $\displaystyle{\sum_{j=1}^{n-1}  j^{\phi(n)} \equiv -1}$ (mod
$n$), where $\phi$ is Euler's totient function (see \cite{BOR}).
\item [iv)] $nB_{\phi(n)} \equiv -1$ (mod $n$), where $B$ is a Bernoulli
number (see \cite{AGO}).
\end{itemize}
\end{prop}

Up to date, only thirteen Giuga Numbers are known (see A007850 in the
\emph{On-Line Encyclopedia of Integer Sequences}):
\begin{itemize}
\item With 3 factors:
$$\textbf{30} = 2 \cdot 3 \cdot 5.$$
\item With 4 factors:
\begin{align*}
\textbf{858}&=2 \cdot 3 \cdot 11 \cdot 13,\\
\textbf{1722}&=2 \cdot 3 \cdot 7 \cdot 41.
\end{align*}
\item With 5 factors:
$$\textbf{66198} = 2 \cdot 3 \cdot 11 \cdot 17 \cdot 59.$$
\item With 6 factors:
\begin{align*}
\textbf{2214408306}&=2 \cdot 3 \cdot 11 \cdot 23 \cdot 31 \cdot 47057,\\
\textbf{24423128562}&=2 \cdot 3 \cdot 7 \cdot 43 \cdot 3041 \cdot 4447.
\end{align*}
\item With 7 factors:
\begin{align*}
\textbf{432749205173838}&=2 \cdot 3  \cdot 7 \cdot 59 \cdot 163 \cdot 1381
\cdot 775807,\\
\textbf{14737133470010574}&=2 \cdot 3 \cdot 7 \cdot 71 \cdot 103 \cdot 67213
\cdot 713863,\\
\textbf{550843391309130318}&=2 \cdot 3 \cdot 7 \cdot 71 \cdot 103 \cdot
61559 \cdot 29133437.
\end{align*}
\item With 8 factors:
\begin{align*}
\textbf{244197000982499715087866346}=\ & 2 \cdot 3 \cdot 11 \cdot 23 \cdot
31 \cdot 47137 \cdot 28282147 \cdot \\ &3892535183,\\
\textbf{554079914617070801288578559178}=\ & 2 \cdot 3 \cdot 11 \cdot 23
\cdot 31 \cdot 47059 \cdot 2259696349 \cdot\\ & 110725121051,\\
\textbf{1910667181420507984555759916338506}=\ & 2 \cdot 3 \cdot 7 \cdot 43\cdots 1831 \cdot 138683 \cdot 2861051 \cdot \\ &1456230512169437.
\end{align*}
\end{itemize}

There are no other Giuga numbers with less than 8 prime factors. There is
another known Giuga Number (found by Frederick Schneider in 2006) which has
10 prime factors, but it is not known if there is any Giuga Number between
this and the previous ones. This biggest known Giuga Number is the following:

\begin{align*}
&\textbf{4200017949707747062038711509670656632404195753751630609228764416}\\
&\textbf{142557211582098432545190323474818}= 2 \cdot 3 \cdot 11 \cdot 23
\cdot 47059 \cdot 2217342227 \cdot\\ &1729101023519 \cdot
8491659218261819498490029296021 \cdot\\
&658254480569119734123541298976556403.
\end{align*}

Observe that all known Giuga Numbers are even. If an odd Giuga number
exists, it must be the product of at least 14 primes. It is not even known
if there are infinitely many Giuga Numbers.

In 2009, Paolo P. Lava, an active collaborator of the \emph{On-Line Encyclopedia of
Integer Sequences}, conjectured that Giuga Numbers were exactly the
solutions of the differential equation $n'=n+1$, with $n'$ being the arithmetic
derivative of $n$. It is not a very known conjecture (see
\cite{LAVA} or comments about sequence A007850 in the \emph{On-Line
Encyclopedia of Integer Sequences}) but it seems known enough to appear in wikipedia's article on Giuga Numbers $$\verb"http://en.wikipedia.org/wiki/Giuga_number".$$

\section{Giuga Numbers and the arithmetic derivative: bringing doubts on Lava's
conjecture.}

It is surprising that Lava's conjecture has not been answered. Certainly the thirteen known Giuga Numbers satisfy $n'=n+1$. Nevertheless, we can observe that this is exclusively due to the fact that these thirteen known Giuga Numbers satisfy $ \sum_{p|n} \frac{1}{p}-\prod_{p|n}
\frac{1}{p} =1$. Let us introduce now a novel characterization of Giuga Numbers in terms of the arithmetic derivative.

To do so we first need the following technical lemma (see \cite{VIC}[Corollary 2]).

\begin{lem} If $n'=an+1$, then $n$ is square-free.
\end{lem}
\begin{proof}
If there is a prime $p$ such that $p^2$ divides $n$, then $p$ must divide $n'$. Since $n'=an+1$ this yields a contradiction.
\end{proof}

\begin{prop} 
Let $n$ be an integer. Then $n$ is a Giuga number if and only if  $n'=an+1$ for some
$a\in\mathbb{N}$
\end{prop}
\begin{proof} 
If $n$ is a Giuga Number then it is composite and square-free.
Thus, we can put $n=p_1 p_2 \cdots p_k$ con $k>1$ and since $n$ is a Giuga Number 
$$ \sum_{i=1}^k \frac{1}{p_i}-\prod_{i=1}^k\frac{1}{p_i}=\sum_{i=1}^k\frac{1}{p_i}
-\frac{1}{n}=a \in \mathbb{Z}.$$
Consequently, $\displaystyle{n\sum_{i=1}^k\frac{1}{p_i} -1=a n}$ and it follows by definition that $n'=a
n+1$.

Conversely, assume that $n'=an+1$ con $a\in \mathbb{N}$. Then $n$ is square-free; i.e., $\displaystyle{\prod_{p|n}\frac{1}{p}=\frac{1}{n}}$. Thus, $\displaystyle{an+1=n'= n\sum_{i=1}^k\frac{1}{p_i}}$ implies that $\displaystyle{\sum_{i=1}^k\frac{1}{p_i}-\frac{1}{n}=a\in \mathbb{N}}$; i.e., $n$ is a Giuga Number, as claimed.
\end{proof}

This result shows that Lava's conjecture is as close (or far away) to be refuted as the discovery of a Giuga Number with 

$$\sum_{i=1}^k\frac{1}{p_i}-\frac{1}{n}>1.$$ 

This is not likely to happen at short term, since it is known that such a number should have more than 59 prime factors \cite{survey}. In any case it has been pointed out that Lava's conjecture is not plausible and it could even be false.

\section{Concerned sequences}
This paper deals with sequences: A003415, A007850.

\end{document}